\documentclass[a4paper]{article}
\usepackage{amsmath}
\usepackage{amssymb}
\usepackage{amsfonts}
\usepackage[english]{babel}
\usepackage{bbm,amsmath,amsthm,tikz}
\usepackage{amscd}

\newcommand{\NN}{\mathbbm{N}}
\newcommand{\RR}{\mathbbm{R}}

\newcommand{\ZZ}{\mathbbm{Z}}

\newtheorem{theorem}{Theorem}
\newtheorem{lemma}[theorem]{Lemma}

\newtheorem{definition}[theorem]{Definition}
\newtheorem{corollary}[theorem]{Corollary}

\DeclareMathOperator{\conv}{conv}
\begin{document}
\title{Convexity of distinct sum sets}
\author{Alexander Lemmens}
\date{\today}
\maketitle
\footnotetext[1]{The author was supported by the Flemish Research Council (FWO) when beginning work on this article.}

\begin{abstract}
We study a combinatorial notion where given a set of lattice points one takes the set of all sums of subsets of a fixed size, and we ask if the given set comes from a convex lattice polytope whether the resulting set also comes from a convex lattice polytope. We obtain a positive result in dimension 2 and a negative result in higher dimensions. We apply this to the corner cut polyhedron.\\

\end{abstract}
\section{Introduction}
\begin{definition}
Let $S$ be a subset of $\ZZ^n$ and $p\in\ZZ$, we define the $p$-th distinct sum set of $S$ as follows:
$$\mathcal{D}_pS:=\{\sum_{m\in T}m|T\subseteq S,\hspace{5 pt}\#T=p\}.$$
When $p=0$ we define $\mathcal{D}_pS$ as $\{0\}$ where $0$ denotes the origin in $\ZZ^n$. If $p<0$ or $p>\#S$ we define $\mathcal{D}_pS$ to be the empty set.
\end{definition}
This is nonempty whenever $0\leq p\leq\# S$. Neither the terminology `distinct sum set', nor the notation are standard in this context, but as far as the author knows this notion doesn't have standard notation or terminology. I call it `distinct sum set' because we are adding together $p$ distinct points of $S$ and taking the set of sums obtained in this way.
In \cite{cuttingcorners,corteel,muller} the authors are concerned with the case where $S$ is $\mathbb{Z}_{\geq 0}^n$, for some $n$, and they call the convex hull of the resulting distinct sum set a corner cut polyhedron. In \cite{wagner} they call the convex hull of $\mathcal{D}_kS$ the $k$-set polytope of $S$ and they denote it $\mathcal{P}_k(S)$. In this article we are mainly concerned with the case where $S$ is the set of lattice points of some convex bounded set in $\RR^n$. We investigate whether the set $\mathcal{D}_pS$ is also the set of lattice points of a convex set. It turns out however that there is a simple counterexample: the set $S:=\{(0,0),(1,0),(0,1),(-1,-1)\}$.
\begin{center}
\begin{tikzpicture}
\draw[<->] (1,0) -- (0,0) -- (0,1);
\draw (0,0) -- (-.8,0);
\draw (0,0) -- (0,-.8);
\draw[thick] (0,.5) -- (.5,0) -- (-.5,-.5) -- (0,.5);
\fill (0,0) circle (1.5 pt);
\fill (0,.5) circle (1.5 pt);
\fill (-.5,-.5) circle (1.5 pt);
\fill (.5,0) circle (1.5 pt);
\end{tikzpicture}
\end{center}
We get $$\mathcal{D}_2S=\{(-1,-1),(-1,0),(0,-1),(1,0),(0,1),(1,1)\},$$
which is not the set of lattice points of anything convex, as $(0,0)$ is missing.
\begin{center}
\begin{tikzpicture}
\draw[<->] (1,0) -- (0,0) -- (0,1);
\draw (0,0) -- (-.8,0);
\draw (0,0) -- (0,-.8);
\fill (.5,.5) circle (1.5 pt);
\fill (0,-.5) circle (1.5 pt);
\fill (-.5,0) circle (1.5 pt);
\fill (0,.5) circle (1.5 pt);
\fill (-.5,-.5) circle (1.5 pt);
\fill (.5,0) circle (1.5 pt);
\end{tikzpicture}
\end{center}


But it turns out the counterexamples are very limited.
Before we can state our theorem we give the following definitions.
\begin{definition}
By a \emph{lattice point} we mean a point in Euclidean space with integer coordinates. A \emph{lattice polytope} is the convex hull of any nonempty finite set of lattice points in Euclidean space. When working in the plane we speak of a \emph{lattice polygon}
\end{definition}
\begin{definition}
We call two lattice polytopes $P,Q$ \emph{equivalent} if there is an affine transformation $T:\RR^n\longrightarrow\RR^n$ with $T(\ZZ^n)=\ZZ^n$ such that $T(P)=Q$.
\end{definition}
By the notation `conv' we will mean the convex hull of a set of points.
\begin{theorem}\label{wedgeconvex}
Let $P\subset\RR^2$ be a convex lattice polygon, not equivalent to a polygon of the form $$\conv\{(-1,-1),(0,1),(a,0)\},$$
with $a\in\ZZ_{>0}$.
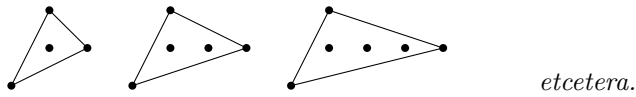
\begin{center}
\begin{tikzpicture}
\draw[<->] (1,0) -- (0,0) -- (0,1);
\draw (0,0) -- (-.8,0);
\draw (0,0) -- (0,-.8);
\draw[thick] (0,.5) -- (.5,0) -- (-.5,-.5) -- (0,.5);
\fill (0,0) circle (1.5 pt);
\fill (0,.5) circle (1.5 pt);
\fill (-.5,-.5) circle (1.5 pt);
\fill (.5,0) circle (1.5 pt);
\end{tikzpicture}
\quad
\begin{tikzpicture}
\draw[<->] (1.5,0) -- (0,0) -- (0,1);
\draw (0,0) -- (-.8,0);
\draw (0,0) -- (0,-.8);
\draw[thick] (0,.5) -- (1,0) -- (-.5,-.5) -- (0,.5);
\fill (0,0) circle (1.5 pt);
\fill (0,.5) circle (1.5 pt);
\fill (-.5,-.5) circle (1.5 pt);
\fill (.5,0) circle (1.5 pt);
\fill (1,0) circle (1.5 pt);
\end{tikzpicture}
\quad
\begin{tikzpicture}
\draw[<->] (2,0) -- (0,0) -- (0,1);
\draw (0,0) -- (-.8,0);
\draw (0,0) -- (0,-.8);
\draw[thick] (0,.5) -- (1.5,0) -- (-.5,-.5) -- (0,.5);
\fill (0,0) circle (1.5 pt);
\fill (0,.5) circle (1.5 pt);
\fill (-.5,-.5) circle (1.5 pt);
\fill (.5,0) circle (1.5 pt);
\fill (1,0) circle (1.5 pt);
\fill (1.5,0) circle (1.5 pt);
\end{tikzpicture}
\quad\quad\quad$\cdots$\newline
\small \textup{exceptions} \normalsize
\end{center}
Then for all $0\leq p\leq N$ the set $\mathcal{D}_p(P\cap\ZZ^2)$ is the set of lattice points of some convex lattice polytope. Here $N=\#(P\cap\ZZ^2)$.
\end{theorem}

\begin{corollary}\label{cornercut}
	For the corner cut polyhedron $P_n^d:=\conv\mathcal{D}_d(\mathbb{N}^n)$ considered in \cite{cuttingcorners} with $n=2$ we have that every lattice point of $P_2^d$ can be written as a sum of $d$ distinct lattice points in $\mathbb{N}^2$.
\end{corollary}
Of course since $\mathbb{R}_{\geq 0}^2$ is unbounded, and hence not conforming to our definition of polytope, we cannot apply the theorem directly, but it easily follows as we will show in section \ref{sec:cornercut}.\\

It turns out that in dimensions higher than two there is little hope of a positive result. In section \ref{3D} we give a counterexample where $P$ is a multiple of the standard simplex in 3D and $p=42$. This will also give a counterexample for $\mathbb{N}^3$ in that not every lattice point in $P_{3}^{42}$ can be written as a sum of 42 distinct lattice points in $\mathbb{N}^3$.\\

The motivation of this research came from studying graded Betti tables of toric surfaces, where each entry in the Betti table has a corresponding bidegree table (see \cite[p.\ 9]{computingbetti} and \cite[p.\ 3]{segre}), and one can ask questions about the convexity of the set of nonzero entries in these bidegree tables.\\

See https://github.com/AlexanderLemmens/DistinctSumSet for an algorithm that computes the distinct sum set.\\

In the next section we prove our convexity result for the two-dimensional case. In section 3 we deal with the corner cut polyhedron with $n=2$. In section 4 we provide a counterexample in 3D.
\subsection*{Acknowledgements}
The author was supported by the Flemish Research Council (FWO) when beginning work on this article. I also want to thank my mentor Wouter Castryck and my colleague Milena Hering for inspiring me to do this research.
\section{Proof of the theorem}
In this section we prove theorem \ref{wedgeconvex}. We first prove some lemmas. Henceforth we will write $\mathcal{D}_pP$ in stead of $\mathcal{D}_p(P\cap\ZZ^n)$. Given a finite set $S\subset\ZZ^n$ we say it is `convex' (with quotation marks) if it is the set of lattice points of some convex lattice polytope.
\begin{lemma}\label{wedgeminus}
Let $P$ be a convex lattice polytope and $N=\#(P\cap\ZZ^2)$. If $\mathcal{D}_pP$ is `convex' then $\mathcal{D}_{N-p}P$ is also `convex'. Here $0\leq p\leq N$.
\end{lemma}
\begin{proof}
Let $u_0$ be the sum of all lattice points of $P$. The result follows from the equality
$$\mathcal{D}_{N-p}P=u_0-\mathcal{D}_pP.$$
\end{proof}
We will use this lemma to reduce to the case where $p\leq N/2$. If $v$ is a vertex of a polytope $P$ we will denote by $P_v$ the convex hull of $P\cap\ZZ^n\backslash\{v\}$.
\begin{lemma}\label{union}
Let $P$ be a convex lattice polytope and $p$ an integer. If $\mathcal{D}_pP_v$ is `convex' for every vertex $v$ of $P$, and
$$\conv\mathcal{D}_pP=\bigcup_{v\textup{ vertex}}\conv\mathcal{D}_pP_v,$$
then $\mathcal{D}_pP$ is `convex'.
\end{lemma}
\begin{proof}
For a set to be `convex' means that every lattice point in its convex hull is an element of the set. So let $m\in\conv\mathcal{D}_pP$ be a lattice point, we prove that it is in $\mathcal{D}_pP$. By the equality $m$ belongs to some $\conv\mathcal{D}_pP_v$. And because $\mathcal{D}_pP_v$ is `convex' it follows that $m\in\mathcal{D}_pP_v$, and so $m\in\mathcal{D}_pP$.
\end{proof}
This will be used to prove the theorem by induction on $N$, the number of lattice points.
\begin{lemma}\label{projection}
Let $P$ be an $n$-dimensional convex lattice polytope with $N$ lattice points and let $1\leq p\leq N-1$. Suppose
$$\bigcap_{v\textup{ vertex}}\conv\mathcal{D}_pP_v\neq\emptyset,$$
and suppose that every facet $E$ of $P$ with at least $N-p+1$ lattice points satisfies
$$\conv\mathcal{D}_{e-N+p}E=\bigcup_{v\textup{ vertex}}\conv\mathcal{D}_{e-N+p}E_v,$$
where $e=\#E\cap\ZZ^n$, then
$$\conv\mathcal{D}_pP=\bigcup_{v\textup{ vertex}}\conv\mathcal{D}_pP_v.$$
\end{lemma}
This lemma allows us to satisfy the requirement of lemma \ref{union}
\begin{proof}
Let $u$ be an element of the set $\bigcap_{v\textup{ vertex}}\conv\mathcal{D}_pP_v$ and let $m_0\in\conv\mathcal{D}_pP$ be a point, we have to prove that $m_0$ is in some $\conv\mathcal{D}_pP_v$. Of course we can suppose that $m_0$ is not equal to $u$, because then the conclusion would be obvious. We claim there exists a facet $F$ of $\conv\mathcal{D}_pP$ such that $m_0$ is in the convex hull of $F\cup\{u\}$. To see this, project $m_0$ away from $u$ onto the boundary of $\conv\mathcal{D}_pP$. Let us call this boundary point $m_1$, so $m_0$ lies on the line segment $[u,m_1]$. Then $m_1$ belongs to some facet $F$ of $\conv\mathcal{D}_pP$, so $m_0\in\conv(\{u\}\cup F)$.\\

Let $\ell:\RR^n\longrightarrow\RR$ be any linear map with the property that $F$ is the set of points in $\conv\mathcal{D}_pP$ where $\ell$ attains its maximum. Call this maximum $c$.\newline For any vertex $w$ of $F$ we know that $w$ can be written as $\sum_{m\in S}m$ for some $S\subset P\cap\ZZ^n$ of size $p$. (This is because the vertex $w$ is an extremal point of $\conv\mathcal{D}_pP$ and hence an element of $\mathcal{D}_pP$.) Let $a$ be the minimal value that $\ell$ attains on $S$, and let $b$ be the number of $m\in S$ with $\ell(m)=a$. Then every point $m'$ of $P\cap\ZZ^n$ with $\ell(m')>a$ will be in $S$, otherwise taking some $m\in S$ with $\ell(m)=a$ and summing over the set $(S\cup\{m'\})\backslash\{m\}$ would yield a point of $\conv\mathcal{D}_pP$ where $\ell$ achieves a greater value than $c=\ell(w)$. Now we have
$$S=(S\cap\ell^{-1}(a))\cup\bigcup_{i\geq a+1}(P\cap\ZZ^n\cap\ell^{-1}(i))\text{, and so}$$
$$p=\#S=b+\sum_{i\geq a+1}\#(P\cap\ZZ^n\cap\ell^{-1}(i))$$
Since $p$ does not depend on the choice of $w$, we conclude that $a$ and $b$ do not depend on the choice of $w$ either. To see this, take a different vertex $w'$ of $F$ that gives values $a',b'$ and suppose for instance that $a'<a$. Now note that the equation $p=b'+\sum_{i\geq a'+1}\#(S\cap\ell^{-1}(i))$ holds. But this sum includes the term $\#(S\cap\ell^{-1}(a))$, which is at least $b$, and $b'>0$, so this expression must be greater than the one with $a$ and $b$ instead of $a'$ and $b'$, which is a contradiction, because both sums equal $p$.
Now there are two cases.\\

\noindent\underline{Case 1: $a$ is not the minimum that $\ell$ attains on $P$.}\newline
In this case let $v$ be any vertex of $P$ with $\ell(v)<a$ (for instance one where $\ell$ attains its minimum on $P$). Then all vertices $w$ of $F$ are contained in $\conv\mathcal{D}_pP_v$, and therefore so is $F$. (The reason $w\in\conv\mathcal{D}_pP_v$ is that $w$ is the sum of points in a set $S$ as above and $v\notin S$ so $S\subset P_v$, and so $w\in\mathcal{D}_p P_v$.) Since $u$ is also in $\conv\mathcal{D}_pP_v$, and $m_0\in\conv(\{u\}\cup F)$ we conclude that $m_0\in\conv\mathcal{D}_pP_v$, as desired.\\

\noindent\underline{Case 2: $a$ is the minimum that $\ell$ attains on $P$.}\newline
In this case $E:=P\cap\ell^{-1}(a)$ is a face of $P$. Let $e=\#E\cap\ZZ^n$ and let $u_1$ be the sum of all lattice points in $P\backslash E$ (there are $N-e$ such points). Then
$$F=u_1+\conv\mathcal{D}_{p+e-N}E.$$
This follows from our analysis of vertices $w$ of $F$, and it also follows that $p+e-N\geq 1$. Since $F$ is $n-1$-dimensional, so is $E$, so $E$ is a facet. This means we can use the hypothesis of the lemma:
$$\conv\mathcal{D}_{e-N+p}E=\bigcup_{v\textup{ vertex}}\conv\mathcal{D}_{e-N+p}E_v.$$
The union is of course over the vertices of $E$. Recall from the beginning of the proof that we projected $m_0$ to the boundary of $\conv\mathcal{D}_P$, yielding a point $m_1\in F$. Now $m_1-u_1$ is in the left hand side of the above equation, hence it belongs to some
$\conv\mathcal{D}_{e-N+p}E_v$. Therefore we have
$$m_1\in u_1+\conv\mathcal{D}_{e-N+p}E_v\subset\conv\mathcal{D}_pP_v.$$
Since $u$ also belongs to the right hand side of this inclusion, and $m_0\in[m_1,u]$, we conclude that $m_0\in\conv\mathcal{D}_pP_v$, finishing the proof.
\end{proof}
As a corollary we obtain
\begin{lemma}\label{intersection}
If $P$ is a convex lattice polygon with $N=\#(P\cap\ZZ^2)\geq 4$ and if $1\leq p\leq N/2$ is an integer such that
$$\bigcap_{v\textup{ vertex}}\conv\mathcal{D}_pP_v\neq\emptyset,$$
then
$$\conv\mathcal{D}_pP=\bigcup_{v\textup{ vertex}}\conv\mathcal{D}_pP_v.$$
\end{lemma}
\begin{proof}
Suppose first that $P$ is two-dimensional.
Let $E$ be a face of $P$ with $e=\#(E\cap\ZZ^2)\geq N-p+1$. If we can prove that
$$\conv\mathcal{D}_{e-N+p}E=\bigcup_{v\textup{ vertex}}\conv\mathcal{D}_{e-N+p}E_v$$
then we can apply lemma \ref{projection} and we are done. Let $p'=e-N+p$, then $1\leq p'\leq e-2$, by the assumptions that $p\leq N/2$ and $N\geq 4$. This also gives $e\geq3$. Of course $E$ is just a line segment and it has exactly two vertices namely the end points. Taking the sum of $p'$ lattice points of this line segment that aren't end points we find an element of $\bigcap_{v\textup{ vertex}}\conv\mathcal{D}_{p'}E_v$, which is therefore non-empty. This actually allows us to apply lemma \ref{projection} to $E$ and $p'$ with $n=1$. all we have to check is that no facet of $E$ has more than $e-p'$ lattice points. But of course facets of $E$ consist of just one point so this is fine. This also deals with the case when $P$ is one-dimensional as we can then apply the same reasoning to $P$ that we applied to $E$.
\end{proof}
\begin{lemma}\label{inclusion}
Let $P$ be a convex lattice polytope and $1\leq p\leq N=\#(P\cap\ZZ^n)$ an integer, then for any bigger polytope $Q$ containing $P$ we have
$$\bigcap_{v\textup{ vertex}}\mathcal{D}_pP_v\subset\bigcap_{v\textup{ vertex}}\mathcal{D}_pQ_v.$$
So if the left hand side is non-empty, then so is the right hand side.
\end{lemma}
\begin{proof}
Let $m\in\bigcap_{v\textup{ vertex}}\mathcal{D}_pP_v$, we prove it is in $\bigcap_{w\textup{ vertex}}\mathcal{D}_pQ_w$. Let $w$ be a vertex of $Q$, we have to show that $m\in\mathcal{D}_pQ_w$. If $w\notin P$ then we have $m\in\mathcal{D}_pP_v\subset\mathcal{D}_pP\subset\mathcal{D}_pQ_w$ and we are done, so suppose $w\in P$. Then $w$ is a vertex of $P$, so $m\in\mathcal{D}_pP_w\subset\mathcal{D}_pQ_w$ and we are done.
\end{proof}
\begin{lemma}\label{good}
If $P$ is a convex lattice polygon with $N=\#(P\cap\ZZ^2)\geq 5$ then $$\bigcap_{v\textup{ vertex}}\mathcal{D}_pP_v\neq\emptyset,$$
for all integers $1\leq p\leq N/2$.
\end{lemma}
\begin{proof}
Let us call a polygon $p$-good if $\bigcap_{v\textup{ vertex}}\mathcal{D}_pP_v\neq\emptyset$. Let us call a polygon good if it is $p$-good for all integers $1\leq p\leq N/2$, so we have to prove that all polygons with at least 5 lattice points are good. Note that if one polygon is $p$-good, then any polygon containing the given one is also $p$-good, by lemma \ref{inclusion}. So if $P$ is a polygon with an odd number of vertices and some $P_v$ is good then $P$ is also good. This means it is enough to prove the lemma for polygons with five lattice points and polygons with an even number of lattice points (at least six).\\

Note that if a polygon has at least $p$ lattice points that aren't vertices then it is $p$-good, because the sum of $p$ lattice points that aren't vertices is in the intersection $\bigcap_{v\textup{ vertex}}\mathcal{D}_pP_v$. So if $P$ has at most $N/2$ vertices then it is good.
We now prove the lemma for polygons with five, six or eight lattice points, by checking it explicitly for those with more than $N/2$ vertices.
We begin with the case $N=5$, at least 3 vertices.
\begin{center}
\begin{tikzpicture}
\draw [<->] (-.5,1.5) -- (-.5,0) -- (1,0);
\draw[thick] (-.5,.5) -- (-.5,0) -- (1,0) -- (-.5,.5);
\fill (0,0) circle (1.5 pt);
\fill (-.5,0) circle (1.5 pt);
\fill (-.5,.5) circle (1.5 pt);
\fill (.5,0) circle (1.5 pt);
\fill (1,0) circle (1.5 pt);
\end{tikzpicture}
\quad\quad
\begin{tikzpicture}
\draw [<->] (-.5,1.5) -- (-.5,0) -- (1,0);
\draw[thick] (-.5,0) -- (.5,0) -- (0,1) -- (-.5,0);
\fill (0,0) circle (1.5 pt);
\fill (0,.5) circle (1.5 pt);
\fill (0,1) circle (1.5 pt);
\fill (.5,0) circle (1.5 pt);
\fill (-.5,0) circle (1.5 pt);
\end{tikzpicture}
\quad\quad
\begin{tikzpicture}
\draw [<->] (-.5,1) -- (-.5,-.5) -- (1,-.5);
\draw[thick] (0,.5) -- (.5,0) -- (0,-.5) -- (-.5,0) -- (0,.5);
\fill (0,0) circle (1.5 pt);
\fill (-.5,0) circle (1.5 pt);
\fill (0,-.5) circle (1.5 pt);
\fill (.5,0) circle (1.5 pt);
\fill (0,.5) circle (1.5 pt);
\end{tikzpicture}
\quad\quad
\begin{tikzpicture}
\draw [<->] (-.5,1) -- (-.5,-.5) -- (1,-.5);
\draw[thick] (-.5,.5) -- (.5,0) -- (0,-.5) -- (-.5,0) -- (-.5,.5);
\fill (0,0) circle (1.5 pt);
\fill (-.5,0) circle (1.5 pt);
\fill (0,-.5) circle (1.5 pt);
\fill (.5,0) circle (1.5 pt);
\fill (-.5,.5) circle (1.5 pt);
\end{tikzpicture}
\quad\quad
\begin{tikzpicture}
\draw [<->] (-.5,1.5) -- (-.5,0) -- (1,0);
\draw[thick] (-.5,.5) -- (-.5,0) -- (.5,0) -- (0,.5)  -- (-.5,.5);
\fill (0,0) circle (1.5 pt);
\fill (-.5,0) circle (1.5 pt);
\fill (-.5,.5) circle (1.5 pt);
\fill (.5,0) circle (1.5 pt);
\fill (0,.5) circle (1.5 pt);
\end{tikzpicture}
\end{center}
Note that all of these are $1$-good, as they all have at least one lattice point that isn't a vertex.
For the first one $(1,0)+(2,0)$ is in the intersection. For the second one $(1,0)+(1,1)$ is in the intersection.
For the third one $(1,0)+(1,2)=(0,1)+(2,1)$ is in the intersection, for the fourth one $(1,0)+(0,2)=(0,1)+(1,1)$ is in the intersection, and for the fifth one $(0,1)+(1,0)=(0,0)+(1,1)$ is in the intersection. In each case we found an element of $\mathcal{D}_2P$ which for any vertex $v$ can always be written as a sum of two distinct lattice points in $P\backslash\{v\}$. So they are all 2-good, and hence good. We now continue with the case $N=6$ where there are at least 4 vertices. The following is a list of all convex lattice polygons with 6 lattice points and at least 4 vertices, up to unimodular equivalence. We label the lattice points with letters from the alphabet.
\begin{center}
\begin{tikzpicture}
\draw (1.5,0) -- (1.5,.5) -- (1,1) -- (0,.5) -- (1.5,0);
\node at (1.5,0) {$a$};
\node at (1,1) {$b$};
\node at (0,.5) {$c$};
\node at (.5,.5) {$d$};
\node at (1,.5) {$e$};
\node at (1.5,.5) {$f$};
\end{tikzpicture}
\quad\quad\quad
\begin{tikzpicture}
\draw (1,0) -- (1.5,.5) -- (1,1) -- (0,.5) -- (1,0);
\node at (1,0) {$a$};
\node at (1,1) {$b$};
\node at (1.5,.5) {$c$};
\node at (1,.5) {$d$};
\node at (.5,.5) {$e$};
\node at (0,.5) {$f$};
\end{tikzpicture}
\quad\quad\quad
\begin{tikzpicture}
\draw (1.5,0) -- (.5,.5) -- (0,0) -- (1,-.5) -- (1.5,0);
\node at (0,0) {$b$};
\node at (1,0) {$d$};
\node at (.5,0) {$c$};
\node at (1.5,0) {$e$};
\node at (1,-.5) {$f$};
\node at (.5,.5) {$a$};
\end{tikzpicture}
\quad\quad\quad
\begin{tikzpicture}
\draw (1.5,0) -- (.5,.5) -- (0,.5) -- (0,0) -- (1.5,0);
\node at (0,0) {$c$};
\node at (1,0) {$e$};
\node at (.5,0) {$d$};
\node at (0,.5) {$a$};
\node at (1.5,0) {$f$};
\node at (.5,.5) {$b$};
\end{tikzpicture}
\end{center}
\begin{center}
\begin{tikzpicture}
\draw (0,1) -- (.5,0) -- (1,.5) -- (1,1) -- (0,1);
\node at (0,1) {$a$};
\node at (.5,0) {$b$};
\node at (.5,.5) {$c$};
\node at (1,.5) {$d$};
\node at (.5,1) {$e$};
\node at (1,1) {$f$};
\end{tikzpicture}
\quad\quad\quad
\begin{tikzpicture}
\draw (0,.5) -- (.5, -.5) -- (1.5,0) -- (.5,.5) -- (0,.5);
\node at (0,.5) {$a$};
\node at (1,0) {$d$};
\node at (.5,0) {$c$};
\node at (1.5,0) {$e$};
\node at (.5,-.5) {$f$};
\node at (.5,.5) {$b$};
\end{tikzpicture}
\quad\quad\quad
\begin{tikzpicture}
\draw (0,.5) -- (0,0) -- (1,0) -- (1,.5) -- (0,.5);
\node at (0,0) {$d$};
\node at (1,0) {$f$};
\node at (.5,0) {$e$};
\node at (0,.5) {$a$};
\node at (1,.5) {$c$};
\node at (.5,.5) {$b$};
\end{tikzpicture}
\quad\quad\quad
\begin{tikzpicture}
\draw (0,.5) -- (.5,0) -- (1,.5) -- (1,1) -- (.5,1) -- (0,.5);
\node at (0,.5) {$a$};
\node at (.5,0) {$b$};
\node at (.5,.5) {$c$};
\node at (1,.5) {$d$};
\node at (.5,1) {$e$};
\node at (1,1) {$f$};
\end{tikzpicture}
\end{center}
All of these polygons are 1-good and 2-good as removing any vertex yields a polygon with five vertices and we already checked that they are 1-good and 2-good. We now show they are all 3-good and hence good. On the top row from left to right we have $d+e+f=d+b+a$, $e+d+c=b+d+a$, $b+c+e=a+c+f$ and $a+d+e=b+c+e$. On the bottom row from left to right we have $e+c+d=e+f+b$, $a+c+e=b+c+d$, $a+b+f=d+b+c$ and $a+c+d=e+c+b$. In each case there is no vertex that appears on both sides of the equation, so they are all 3-good, and hence good. Finally we move on to the case $N=8$. Here is a comprehensive list of all convex lattice polygons with 8 lattice points and at least 5 vertices, up to unimodular equivalence.
\begin{center}
\begin{tikzpicture}
\draw (1,1) -- (0,.5) -- (1,0) -- (1.5,0) -- (1.5,1) -- (1,1);
\node at (1,1) {$a$};
\node at (1.5,1) {$b$};
\node at (0,.5) {$c$};
\node at (.5,.5) {$d$};
\node at (1,.5) {$e$};
\node at (1.5,.5) {$f$};
\node at (1,0) {$g$};
\node at (1.5,0) {$h$};
\end{tikzpicture}
\quad\quad\quad
\begin{tikzpicture}
\draw (.5,1) -- (0,.5) -- (.5,0) -- (1,0) -- (1.5,1) -- (.5,1);
\node at (.5,1) {$a$};
\node at (1,1) {$b$};
\node at (1.5,1) {$c$};
\node at (0,.5) {$d$};
\node at (.5,.5) {$e$};
\node at (1,.5) {$f$};
\node at (.5,0) {$g$};
\node at (1,0) {$h$};
\end{tikzpicture}
\quad\quad\quad
\begin{tikzpicture}
\draw (0,1) -- (0,0) -- (1,0) -- (1,.5) -- (.5,1) -- (0,1);
\node at (0,1) {$a$};
\node at (.5,1) {$b$};
\node at (0,.5) {$c$};
\node at (.5,.5) {$d$};
\node at (1,.5) {$e$};
\node at (0,0) {$f$};
\node at (.5,0) {$g$};
\node at (1,0) {$h$};
\end{tikzpicture}
\end{center}
\begin{center}
\begin{tikzpicture}
\draw (.5,1) -- (0,.5) -- (.5,0) -- (1.5,.5) -- (1.5,1) -- (.5,1);
\node at (.5,1) {$a$};
\node at (1,1) {$b$};
\node at (1.5,1) {$c$};
\node at (0,.5) {$d$};
\node at (.5,.5) {$e$};
\node at (1,.5) {$f$};
\node at (1.5,.5) {$g$};
\node at (.5,0) {$h$};
\end{tikzpicture}
\quad\quad\quad
\begin{tikzpicture}
\draw (.5,1) -- (0,.5) -- (1,0) -- (1.5,.5) -- (1.5,1) -- (.5,1);
\node at (.5,1) {$a$};
\node at (1,1) {$b$};
\node at (1.5,1) {$c$};
\node at (0,.5) {$d$};
\node at (.5,.5) {$e$};
\node at (1,.5) {$f$};
\node at (1.5,.5) {$g$};
\node at (1,0) {$h$};
\end{tikzpicture}
\quad\quad\quad
\begin{tikzpicture}
\draw (.5,1) -- (0,.5) -- (.5,0) -- (1,0) -- (1.5,.5) -- (1,1) -- (.5,1);
\node at (.5,1) {$a$};
\node at (1,1) {$b$};
\node at (0,.5) {$c$};
\node at (.5,.5) {$d$};
\node at (1,.5) {$e$};
\node at (1.5,.5) {$f$};
\node at (.5,0) {$g$};
\node at (1,0) {$h$};
\end{tikzpicture}
\end{center}
\begin{center}
\begin{tikzpicture}
\draw (.5,1) -- (0,.5) -- (0,0) -- (.5,0) -- (1.5,.5) -- (1,1) -- (.5,1);
\node at (.5,1) {$a$};
\node at (1,1) {$b$};
\node at (0,.5) {$c$};
\node at (.5,.5) {$d$};
\node at (1,.5) {$e$};
\node at (1.5,.5) {$f$};
\node at (0,0) {$g$};
\node at (.5,0) {$h$};
\end{tikzpicture}
\quad\quad\quad
\begin{tikzpicture}
\draw (1.5,1) -- (0,.5) -- (1,0) -- (1.5,0) -- (2,1) -- (1.5,1);
\node at (1.5,1) {$a$};
\node at (2,1) {$b$};
\node at (0,.5) {$c$};
\node at (.5,.5) {$d$};
\node at (1,.5) {$e$};
\node at (1.5,.5) {$f$};
\node at (1,0) {$g$};
\node at (1.5,0) {$h$};
\end{tikzpicture}
\quad\quad\quad
\begin{tikzpicture}
\draw (2,1) -- (0,.5) -- (1,0) -- (1.5,0) -- (2,.5) -- (2,1);
\node at (2,1) {$a$};
\node at (0,.5) {$b$};
\node at (.5,.5) {$c$};
\node at (1,.5) {$d$};
\node at (1.5,.5) {$e$};
\node at (2,.5) {$f$};
\node at (1,0) {$g$};
\node at (1.5,0) {$h$};
\end{tikzpicture}
\end{center}
\begin{center}
\begin{tikzpicture}
\draw (1.5,1) -- (0,.5) -- (1,0) -- (1.5,0) -- (2,.5) -- (1.5,1);
\node at (1.5,1) {$a$};
\node at (0,.5) {$b$};
\node at (.5,.5) {$c$};
\node at (1,.5) {$d$};
\node at (1.5,.5) {$e$};
\node at (2,.5) {$f$};
\node at (1,0) {$g$};
\node at (1.5,0) {$h$};
\end{tikzpicture}
\quad\quad\quad
\begin{tikzpicture}
\draw (1,1) -- (0,.5) -- (1,0) -- (1.5,0) -- (2,.5) -- (1,1);
\node at (1,1) {$a$};
\node at (0,.5) {$b$};
\node at (.5,.5) {$c$};
\node at (1,.5) {$d$};
\node at (1.5,.5) {$e$};
\node at (2,.5) {$f$};
\node at (1,0) {$g$};
\node at (1.5,0) {$h$};
\end{tikzpicture}
\quad\quad\quad
\begin{tikzpicture}
\draw (1.5,1.5) -- (.5,1) -- (0,.5) -- (1,0) -- (1.5,1) -- (1.5,1.5);
\node at (1.5,1.5) {$a$};
\node at (.5,1) {$b$};
\node at (1,1) {$c$};
\node at (1.5,1) {$d$};
\node at (0,.5) {$e$};
\node at (.5,.5) {$f$};
\node at (1,.5) {$g$};
\node at (1,0) {$h$};
\end{tikzpicture}
\end{center}
All of these polygons are 1-good, 2-good and 3-good, as removing two vertices gives a polygon with six vertices, for which we already proved this. We have to prove that they are 4-good. The polygons on the first two rows all contain something equivalent to the polygon
\begin{center}
\begin{tikzpicture}
\draw (0,.5) -- (0,0) -- (1,0) -- (1,.5) -- (0,.5);
\node at (0,0) {$d$};
\node at (1,0) {$f$};
\node at (.5,0) {$e$};
\node at (0,.5) {$a$};
\node at (1,.5) {$c$};
\node at (.5,.5) {$b$};
\end{tikzpicture}
\end{center}
which is 4-good as $a+b+e+f=b+c+d+e$, so on the first and the second row all polygons are good. On the third row from left to right we have: $a+d+e+g=b+c+d+h$, $b+e+f+g=a+e+f+h$ and $a+c+e+g=c+d+e+f$. On the fourth row we have $a+c+e+h=c+d+e+f$, $b+e+f+g=c+d+e+h$ and $a+c+e+h=b+d+f+g$. So they are all good.\\

Now we prove that all polygons with more than eight lattice points are good. We do so by induction. So let $P$ be a polygon with $N=\#(P\cap\ZZ^2)\geq 9$. If $N$ is odd then we simply remove a vertex and by induction the resulting polygon is good so $P$ is also good. So suppose $N\geq10$ is even. Again by removing a vertex and applying the induction hypothesis we conclude that $P$ is $p$-good for all $1\leq p\leq N/2-1$, so we only have to prove that $P$ is $p$-good with $p=N/2$. If $P$ has at most $N/2$ vertices we are done because we can then just take the sum of $p$ lattice points that aren't vertices and this will be in $\bigcap_{v\text{ vertex}}\mathcal{D}_pP_v$. So suppose $P$ has at least $N/2+1$ vertices. Then $P$ must have at least one edge with only two lattice points. Let $E$ be such an edge. Take a unimodular transformation so that $E=[(0,0),(1,0)]$ and $P\subset \RR\times\RR_{\geq 0}$. Then $P$ contains some point of the form $(t,1)$. So $t\in[a,a+1]$ for some integer $a$. Applying a transformation of the form $(x,y)\mapsto(x-ay,y)$ we get $t\in[0,1]$. Now either $P$ contains a lattice point of the form $(b,c)$ with $b\leq 0$, $c\geq 1$ or it contains a lattice point of the form $(b,c)$ with $b,c\geq 1$. In the first case $(0,1)\in P$, as $(0,1)$ would be in the convex hull of $(0,0)$, $(t,1)$ and $(b,c)$. In the second case $(1,1)\in P$ as $(1,1)$ would be in the convex hull of $(1,0)$, $(t,1)$ and $(b,c)$. So there exists at least one integer $b$ such that $(b,1)\in P$. Let $b_0$ be the smallest such integer and $b_1$ the greatest.\\

\noindent\underline{Case 1: $b_0<b_1$}\newline In this case we set
$$Q=\conv(P\cap\ZZ^2\backslash\{(0,0),(1,0),(b_0,1),(b_0+1,1)\}),$$
which is a subset of $P$ with $N-4$ lattice points, so by induction it is $(N-4)/2$-good. Said differently, it is $p-2$-good. Let $u\in\bigcap_{v\text{ vertex}}\mathcal{D}_{p-2}Q_v$. We claim that $u':=u+(b_0+1,1)\in\bigcap_{v\text{ vertex}}\mathcal{D}_pP_v$ so that $P$ is $p$-good and hence good. To prove the claim, let $v$ be any vertex of $P$. If $v=(0,0)$ or $v=(b_0+1,1)$ we can write $u'$ as the sum of $(1,0)$, $(b_0,1)$ and $u$, which in turn is a sum of $p-2$ distinct lattice points of $Q$, so $u'\in\mathcal{D}_pP_v$. If $v=(1,0)$ or $v=(b_0,1)$ then we can write $u'$ as the sum of $(0,0)$, $(b_0+1,1)$ and $u$ which in turn is a sum of $p-2$ distinct lattice points of $Q$, so $u'\in\mathcal{D}_pP_v$. If $v$ is none of the above then $v$ is in fact a vertex of $Q$ and we can write $u'$ as the sum of $(1,0)$, $(b_0,1)$ and $u$, which in turn is a sum of $p-2$ distinct lattice points of $Q_v$, so $u'\in\mathcal{D}_pP_v$. We conclude that $u'\in\bigcap_{v\text{ vertex}}\mathcal{D}_pP_v$, so $P$ is $p$-good and hence good.\\

\noindent\underline{Case 2: $b_0=b_1$}\newline
In this case $P$ has only one lattice point of the form $(b,1)$ with $b$ an integer. Using the transformation $(x,y)\mapsto(x+y-by,y)$ we can suppose that $(1,1)\in P$. Since $(0,1)$ and $(2,1)$ are not in $P$ the point $(1,1)$ can't possibly be a vertex of $P$. It also follows that $P\backslash[(0,0),(1,0)]$ is contained in the region of points $(x,y)$ satisfying $0<x<y+1$, $y>0$. Now if $(1,2)\notin P$ then $P\backslash[(0,0),(1,0)]$ is contained in the region of points $(x,y)$ satisfying $x>y/2$, and likewise if $(2,2)\notin P$ then $P$ is contained in the region of points $(x,y)$ satisfying $x<y/2+1$. These can't both be the case, because then $P$ would be contained in such a narrow region that it would have to be a triangle with vertices $(0,0)$, $(1,0)$, $(c,2c-1)$, for some positive integer $c$, but this is excluded as we are assuming $P$ to have at least $N/2+1\geq 6$ vertices. So either $(1,2)\in P$ or $(2,2)\in P$ or both. Using the transformation $(x,y)\mapsto(1+y-x,y)$ we can assume $(1,2)\in P$.\\

Suppose $(1,2)$ is not a vertex of $P$. Then we define the polygon
$$Q=\conv(P\cap\ZZ^2\backslash\{(0,0),(1,0),(1,1),(1,2)\})$$
which has $N-4$ lattice points. By the induction hypothesis $Q$ is $p-2$-good with $p=N/2$. Let $u\in\bigcap_{v\text{ vertex}}\mathcal{D}_{p-2}Q_v$. We claim that $u':=u+(2,3)\in\bigcap_{v\text{ vertex}}\mathcal{D}_pP_v$ so that $P$ is $p$-good and hence good. Let $v$ be a vertex of $P$. By our assumption and what we have seen before $v$ cannot be $(1,1)$ or $(1,2)$. Now $u'$ can be written as the sum of $(1,1)$, $(1,2)$ and $u$ which in turn is a sum of $p-2$ lattice points of $Q$ distinct from $v$. So $u'\in\bigcap_{v\text{ vertex}}\mathcal{D}_pP_v$, so $P$ is $p$-good and hence good.\\

So we can assume that $(1,2)$ is a vertex of $P$. If $(2,2)$ happens to be a point of $P$ but not a vertex we can apply the same reasoning to conclude that $P$ is good, so we can also assume that $(2,2)$ is either not in $P$, or it is a vertex of $P$. However both cases will lead to contradiction. In the first case $P$ will be a quadrangle with vertices $(0,0)$, $(0,1)$, $(1,2)$ and $(c,2c-1)$ and in the second case it will be a pentagon with vertices $(0,0)$, $(0,1)$, $(1,2)$, $(2,2)$ and $(c,2c-1)$ (for some positive integer $c$). Both of these contradict our assumption that $P$ has at least $N/2+1\geq 6$ vertices.
\end{proof}
\begin{proof}[Proof of theorem \ref{wedgeconvex}]
We prove this by induction on $N$, the number lattice points. Note that the cases $p=0$ and $p=1$ are always trivial, and that we can always reduce to the case $p\leq N/2$ using lemma \ref{wedgeminus}. We begin with the case where $N\leq 4$ but $P$ is not the exception with four lattice points. Then $P$ is either a single point, a line with two lattice points, a line with three lattice points, a line with four lattice points, something equivalent to the triangle with vertices $(0,0)$, $(1,0)$, $(0,1)$, or the triangle with lattice points $(1,1)$, $(0,0)$, $(1,0)$, $(2,0)$, or something equivalent to the square $[0,1]\times[0,1]$. In case $P$ is a line with $N$ lattice points $\mathcal{D}_pP$ is a line with $\binom{N}{p}$ lattice points. If $P$ is the square then $\mathcal{D}_2P$ is $\{(1,0),(0,1),(1,1),(1,2),(2,1)\}$, which is `convex'. If $P$ is the triangle with four vertices then $\mathcal{D}_2P$ is $\{(1,0),(1,1),(2,0),(2,1),(3,0),(3,1)\}$. These are all `convex'. Everything else reduces to the case $p=1$ or $p=0$ using lemma \ref{wedgeminus}. This deals with the case $N\leq 4$. We henceforth assume $N\geq5$.\\

The whole idea of the proof is that given a polygon $P$ with $N\geq 5$ and a positive integer $p\leq N/2$ if the theorem is true for all the $\mathcal{D}_p P_v$ then the theorem is true for $\mathcal{D}_pP$. This is because by lemma \ref{good} the condition of lemma \ref{intersection} is satisfied which in turn means the condition of lemma \ref{union} is satisfied. One can always reduce to the case $p\leq N/2$ using lemma \ref{wedgeminus}. Because of this reasoning we already know by induction that the theorem is true for polygons that don't contain any of the exceptional polygons.\\

We now prove that even for the exceptions $\mathcal{D}_pP$ will be `convex' if $p\notin\{2,N-2\}$ and $N\geq 6$. As always we assume $p\leq N/2$ and $p\notin\{0,1\}$. So $p\geq 3$. We prove it by induction starting with the case $N=6$,
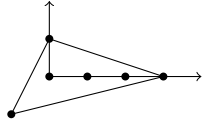
\begin{center}
\begin{tikzpicture}
\draw [<->] (2,0) -- (0,0) -- (0,1);
\draw (0,0) -- (-.8,0);
\draw (0,0) -- (0,-.8);
\draw (0,.5) -- (1.5,0) -- (-.5,-.5) -- (0,.5);
\fill (0,0) circle (1.5 pt);
\fill (0,.5) circle (1.5 pt);
\fill (1,0) circle (1.5 pt);
\fill (1.5,0) circle (1.5 pt);
\fill (-.5,-.5) circle (1.5 pt);
\fill (.5,0) circle (1.5 pt);
\end{tikzpicture}
\end{center}
The only case to look at is $p=3$. Taking $\mathcal{D}_3$ of this polygon yields the set
\begin{center}
\begin{tikzpicture}
\draw [<->] (3.5,0) -- (0,0) -- (0,1);
\draw (0,0) -- (-.8,0);
\draw (0,0) -- (0,-.8);
\fill (.5,.5) circle (1.5 pt);
\fill (1,.5) circle (1.5 pt);
\fill (1.5,.5) circle (1.5 pt);
\fill (2,.5) circle (1.5 pt);
\fill (2.5,.5) circle (1.5 pt);
\fill (-.5,0) circle (1.5 pt);
\fill (0,0) circle (1.5 pt);
\fill (.5,0) circle (1.5 pt);
\fill (1,0) circle (1.5 pt);
\fill (1.5,0) circle (1.5 pt);
\fill (2,0) circle (1.5 pt);
\fill (2.5,0) circle (1.5 pt);
\fill (3,0) circle (1.5 pt);
\fill (0,-.5) circle (1.5 pt);
\fill (.5,-.5) circle (1.5 pt);
\fill (1,-.5) circle (1.5 pt);
\fill (1.5,-.5) circle (1.5 pt);
\fill (2,-.5) circle (1.5 pt);
\end{tikzpicture},
\end{center}
which is `convex'. Now suppose $N$ is at least seven. It is enough to prove it for $3\leq p\leq N/2$. But upon removing any vertex from $\conv\{(-1,-1),(0,1),(N-3,0)\}$ we either get something that does not contain any exceptional polygon, for which we already proved the theorem, or we get the exceptional polygon with $N-1$ lattice points, for which we know that the $\mathcal{D}_p$ is `convex' by induction.\\

Now we prove the theorem for the remaining polygons. So $N\geq 5$ and $P$ is not equivalent to one of the exceptions. Again we work with induction and we assume $p\leq N/2$, so that it is enough if $\mathcal{D}_pP_v$ is `convex' for every vertex $v$ of $P$. This will be the case by the induction hypothesis, except when $P_v$ is one of the exceptions. Suppose $P_v$ is one of the exceptions and assume at first that $N\geq 7$. If $P_v$ has at least six lattice points then by the above argument $\mathcal{D}_pP_v$ is convex if $p$ is not 2 or $N-3$. The inequalities $p\leq N/2$ and $N\geq 7$ ensure that $p\neq N-3$.
So if $N\geq 7$ the only thing left to check is that $\mathcal{D}_2P$ is `convex'. So $P_v=\conv\{(-1,-1),(0,1),(N-4,0)\}$ and $v$ can only be $(-1,0)$, $(0,-1)$ or $(1,1)$. Now any lattice point of $\conv\mathcal{D}_2P$ belongs to $2P$ and can hence be written as a sum of two lattice points of $P$. The only thing that can go wrong is if those two lattice points of $P$ are equal. So we have to show that if $m\in P\cap\ZZ^2$ and $2m\in\conv\mathcal{D}_2P$ then $2m$ is the sum of two distinct points in $P\cap\ZZ^2$. Because $2m\in\conv\mathcal{D}_2P$, $m$ cannot be a vertex of $P$. Therefore $m$ is of the form $(a,0)$ with $0\leq a\leq N-5$, and one easily checks that regardless of which element of $\{(-1,0),(0,-1),(1,1)\}$ $v$ is equal to, we can always write $2m$ as the sum of two distinct elements of $P\cap\ZZ^2$.\\

All that remains to be done is to check the theorem when $N\leq 6$ in the case where for some vertex $v$ of $P$ we have that $P_v$ is one of the exceptions. Up to equivalence the following are the only possibilities for $P$:
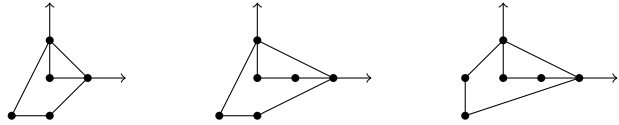
\begin{center}
\begin{tikzpicture}
\draw [<->] (1,0) -- (0,0) -- (0,1);
\draw (0,0) -- (-.8,0);
\draw (0,0) -- (0,-.8);
\draw (0,.5) -- (.5,0) -- (0,-.5) -- (-.5,-.5) -- (0,.5);
\fill (0,0) circle (1.5 pt);
\fill (0,.5) circle (1.5 pt);
\fill (0,-.5) circle (1.5 pt);
\fill (-.5,-.5) circle (1.5 pt);
\fill (.5,0) circle (1.5 pt);
\end{tikzpicture}
\quad\quad\quad
\begin{tikzpicture}
\draw [<->] (1.5,0) -- (0,0) -- (0,1);
\draw (0,0) -- (-.8,0);
\draw (0,0) -- (0,-.8);
\draw (0,.5) -- (1,0) -- (0,-.5) -- (-.5,-.5) -- (0,.5);
\fill (0,0) circle (1.5 pt);
\fill (0,.5) circle (1.5 pt);
\fill (1,0) circle (1.5 pt);
\fill (0,-.5) circle (1.5 pt);
\fill (-.5,-.5) circle (1.5 pt);
\fill (.5,0) circle (1.5 pt);
\end{tikzpicture}
\quad\quad\quad
\begin{tikzpicture}
\draw [<->] (1.5,0) -- (0,0) -- (0,1);
\draw (0,0) -- (-.8,0);
\draw (0,0) -- (0,-.8);
\draw (0,.5) -- (1,0) -- (-.5,-.5) -- (-.5,0) -- (0,.5);
\fill (0,0) circle (1.5 pt);
\fill (0,.5) circle (1.5 pt);
\fill (1,0) circle (1.5 pt);
\fill (-.5,0) circle (1.5 pt);
\fill (-.5,-.5) circle (1.5 pt);
\fill (.5,0) circle (1.5 pt);
\end{tikzpicture}
\end{center}
The case $p=2$ follows by the exact same argument as above. For the first one this is all we have to check as $N=5$ and $p\leq N/2$. For the other two we also have to check that $\mathcal{D}_3P$ is `convex'. Calculating this for these two polygons we get the following two sets,
\begin{center}
\begin{tikzpicture}
\draw [<->] (2,0) -- (0,0) -- (0,1);
\draw (0,0) -- (-1,0);
\draw (0,0) -- (0,-1.5);
\fill (.5,.5) circle (1.5 pt);
\fill (1,.5) circle (1.5 pt);
\fill (1.5,.5) circle (1.5 pt);
\fill (-.5,0) circle (1.5 pt);
\fill (0,0) circle (1.5 pt);
\fill (.5,0) circle (1.5 pt);
\fill (1,0) circle (1.5 pt);
\fill (1.5,0) circle (1.5 pt);
\fill (0,-.5) circle (1.5 pt);
\fill (.5,-.5) circle (1.5 pt);
\fill (1.5,-.5) circle (1.5 pt);
\fill (1,-.5) circle (1.5 pt);
\fill (-.5,-.5) circle (1.5 pt);
\fill (-.5,-1) circle (1.5 pt);
\fill (0,-1) circle (1.5 pt);
\fill (.5,-1) circle (1.5 pt);
\end{tikzpicture}
\quad\quad\quad
\begin{tikzpicture}
\draw [<->] (2,0) -- (0,0) -- (0,1);
\draw (0,0) -- (-1.5,0);
\draw (0,0) -- (0,-1);
\fill (.5,.5) circle (1.5 pt);
\fill (1,.5) circle (1.5 pt);
\fill (1.5,.5) circle (1.5 pt);
\fill (0,.5) circle (1.5 pt);
\fill (-.5,.5) circle (1.5 pt);
\fill (-.5,0) circle (1.5 pt);
\fill (0,0) circle (1.5 pt);
\fill (.5,0) circle (1.5 pt);
\fill (1,0) circle (1.5 pt);
\fill (1.5,0) circle (1.5 pt);
\fill (-1,0) circle (1.5 pt);
\fill (0,-.5) circle (1.5 pt);
\fill (.5,-.5) circle (1.5 pt);
\fill (1,-.5) circle (1.5 pt);
\fill (-.5,-.5) circle (1.5 pt);
\fill (-1,-.5) circle (1.5 pt);
\end{tikzpicture}
\end{center}
and these are indeed `convex'. This concludes the proof.
\end{proof}
\section{The corner cut polyhedron}\label{sec:cornercut}
For $d\geq 0$ the 2D corner cut polyhedron is defined as
$$P_2^d=\conv\mathcal{D}_d\NN^2.$$
(By $\NN$ we mean $\ZZ_{\geq 0}$.) We will now prove corollary \ref{cornercut} which says that every lattice point of $P_2^d$ is a sum of $d$ distinct lattice points in $\NN^2$.
\begin{proof}[Proof of corollary \ref{cornercut}]
	Let $m$ be a lattice point in $P_2^d$. We prove that it is in $\mathcal{D}_d\NN^2$.
	Let $m_1,\ldots,m_k$ be elements of $\mathcal{D}_d\NN^2$ such that $m\in\conv\{m_1,\ldots,m_k\}$. Let $S$ be any finite subset of $\NN^2$ such that each $m_i$ is a sum of $d$ distinct elements of $S$. Let $P$ be any (bounded) convex lattice polygon contained in $\RR_{\geq 0}^2$ that contains $S$. We then have that $m_1,\ldots,m_k\in\mathcal{D}_d(P\cap\ZZ^2)$, and hence $m\in\conv\mathcal{D}_d(P\cap\ZZ^2)$. We have to take $P$ so that it isn't equivalent to any of the exceptions in theorem \ref{wedgeconvex}. For this it is enough if $P$ contains the points $(0,0)$, $(0,1)$, $(1,0)$ and $(1,1)$. By the theorem we then have $m\in\mathcal{D}_d(P\cap\ZZ^2)$ and hence $m\in\mathcal{D}_d\NN^2$.
\end{proof}
\section{A 3D counterexample}\label{3D}
It turns out that in three dimensions convexity fails even for
$$\mathcal{D}_{42}(\ZZ^3\cap\conv((0,0,0),(0,0,6),(0,6,0),(6,0,0))).$$
To see this consider the following diagram of the lattice points in
$$P:=\conv((0,0,0),(0,0,6),(0,6,0),(6,0,0)):$$
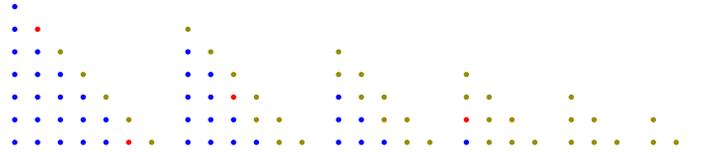
\begin{center}
\begin{tikzpicture}
\fill[blue] (0,0) circle (1 pt);
\fill[blue] (.3,0) circle (1 pt);
\fill[blue] (0,.3) circle (1 pt);
\fill[blue] (.6,0) circle (1 pt);
\fill[blue] (.3,.3) circle (1 pt);
\fill[blue] (0,.6) circle (1 pt);
\fill[blue] (.9,0) circle (1 pt);
\fill[blue] (.6,.3) circle (1 pt);
\fill[blue] (.3,.6) circle (1 pt);
\fill[blue] (0,.9) circle (1 pt);
\fill[blue] (1.2,0) circle (1 pt);
\fill[blue] (.9,.3) circle (1 pt);
\fill[blue] (.6,.6) circle (1 pt);
\fill[blue] (.3,.9) circle (1 pt);
\fill[blue] (0,1.2) circle (1 pt);
\fill[red] (1.5,0) circle (1 pt);
\fill[blue] (1.2,.3) circle (1 pt);
\fill[blue] (.9,.6) circle (1 pt);
\fill[blue] (.6,.9) circle (1 pt);
\fill[blue] (.3,1.2) circle (1 pt);
\fill[blue] (0,1.5) circle (1 pt);
\fill[olive] (1.8,0) circle (1 pt);
\fill[olive] (1.5,.3) circle (1 pt);
\fill[olive] (1.2,.6) circle (1 pt);
\fill[olive] (.9,.9) circle (1 pt);
\fill[olive] (.6,1.2) circle (1 pt);
\fill[red] (.3,1.5) circle (1 pt);
\fill[blue] (0,1.8) circle (1 pt);
\end{tikzpicture}
\hspace{5 pt}
\begin{tikzpicture}
\fill[blue] (0,0) circle (1 pt);
\fill[blue] (.3,0) circle (1 pt);
\fill[blue] (0,.3) circle (1 pt);
\fill[blue] (.6,0) circle (1 pt);
\fill[blue] (.3,.3) circle (1 pt);
\fill[blue] (0,.6) circle (1 pt);
\fill[blue] (.9,0) circle (1 pt);
\fill[blue] (.6,.3) circle (1 pt);
\fill[blue] (.3,.6) circle (1 pt);
\fill[blue] (0,.9) circle (1 pt);
\fill[olive] (1.2,0) circle (1 pt);
\fill[olive] (.9,.3) circle (1 pt);
\fill[red] (.6,.6) circle (1 pt);
\fill[blue] (.3,.9) circle (1 pt);
\fill[blue] (0,1.2) circle (1 pt);
\fill[olive] (1.5,0) circle (1 pt);
\fill[olive] (1.2,.3) circle (1 pt);
\fill[olive] (.9,.6) circle (1 pt);
\fill[olive] (.6,.9) circle (1 pt);
\fill[olive] (.3,1.2) circle (1 pt);
\fill[olive] (0,1.5) circle (1 pt);
\end{tikzpicture}
\hspace{5 pt}
\begin{tikzpicture}
\fill[blue] (0,0) circle (1 pt);
\fill[blue] (.3,0) circle (1 pt);
\fill[blue] (0,.3) circle (1 pt);
\fill[blue] (.6,0) circle (1 pt);
\fill[blue] (.3,.3) circle (1 pt);
\fill[blue] (0,.6) circle (1 pt);
\fill[olive] (.9,0) circle (1 pt);
\fill[olive] (.6,.3) circle (1 pt);
\fill[olive] (.3,.6) circle (1 pt);
\fill[olive] (0,.9) circle (1 pt);
\fill[olive] (1.2,0) circle (1 pt);
\fill[olive] (.9,.3) circle (1 pt);
\fill[olive] (.6,.6) circle (1 pt);
\fill[olive] (.3,.9) circle (1 pt);
\fill[olive] (0,1.2) circle (1 pt);
\end{tikzpicture}
\hspace{5 pt}
\begin{tikzpicture}
\fill[blue] (0,0) circle (1 pt);
\fill[olive] (.3,0) circle (1 pt);
\fill[red] (0,.3) circle (1 pt);
\fill[olive] (.6,0) circle (1 pt);
\fill[olive] (.3,.3) circle (1 pt);
\fill[olive] (0,.6) circle (1 pt);
\fill[olive] (.9,0) circle (1 pt);
\fill[olive] (.6,.3) circle (1 pt);
\fill[olive] (.3,.6) circle (1 pt);
\fill[olive] (0,.9) circle (1 pt);
\end{tikzpicture}
\hspace{5 pt}
\begin{tikzpicture}
\fill[olive] (0,0) circle (1 pt);
\fill[olive] (.3,0) circle (1 pt);
\fill[olive] (0,.3) circle (1 pt);
\fill[olive] (.6,0) circle (1 pt);
\fill[olive] (.3,.3) circle (1 pt);
\fill[olive] (0,.6) circle (1 pt);
\end{tikzpicture}
\hspace{5 pt}
\begin{tikzpicture}
\fill[olive] (0,0) circle (1 pt);
\fill[olive] (.3,0) circle (1 pt);
\fill[olive] (0,.3) circle (1 pt);
\end{tikzpicture}
\hspace{5 pt}
\begin{tikzpicture}
\fill[olive] (0,0) circle (1 pt);
\end{tikzpicture}
\end{center}
The first triangle consists of the points with third coordinate equal to zero, the second triangle consists of the points with third coordinate equal to one, etcetera.
Of the 84 points 40 are blue, 40 are olive green and 4 are red. In fact the four red points span a plane that separates the blue points from the olive green points. The four red points have coordinates $P_1:=(5,0,0)$, $P_2:=(1,5,0)$, $P_3:=(2,2,1)$, $P_4:=(0,1,3)$ respectively. One can see that $P_1+P_2+P_4=3P_3$, so they indeed span a plane. In fact the simplest counterexample to convexity of the distinct sum set in 2 dimensions also has four lattice points and that configuration is equivalent to the one of the four red points in this setting. Therefore the set $\mathcal{D}_2\{P_1,P_2,P_3,P_4\}$ is not the set of lattice points of a convex set. Now the sum of the blue points plus $\mathcal{D}_2\{P_1,P_2,P_3,P_4\}$ will span a (2D) facet of $\conv\mathcal{D}_{42}P$, and the intersection of this facet with $\mathcal{D}_{42}P$ is exactly the sum of the blue points plus $\mathcal{D}_2\{P_1,P_2,P_3,P_4\}$. Since $\mathcal{D}_2\{P_1,P_2,P_3,P_4\}$ is not `convex', neither is $\mathcal{D}_{42}P$. To be specific, the sum of all the blue points plus $2P_3$ is in $\conv\mathcal{D}_{42}P$ but not in $\mathcal{D}_{42}P$, as $2P_3$ can not be written as the sum of two distinct red points. Note that when replacing $P$ with $\mathbb{N}^3$ the same counterexample still works.

\vspace{10 pt}
\large
Alexander Lemmens\newline
COSIC, Kasteelpark Arenberg 10,\newline
B-3001 Heverlee-Leuven,\newline
Belgium
\end{document}